\documentclass{amsart}
\usepackage{amsmath,amssymb}
\usepackage{graphicx,subfigure}

\usepackage{color}

\newtheorem{theorem}{Theorem}[section]
\newtheorem{proposition}[theorem]{Proposition}
\newtheorem{corollary}[theorem]{Corollary}

\newtheorem{lemma}[theorem]{Lemma}

\theoremstyle{definition}

\newtheorem*{definition*}{Definition}

\theoremstyle{remark}
\newtheorem{remark}[theorem]{Remark}

\numberwithin{equation}{section}

\newcommand{\al}{\alpha}

\newcommand{\ep}{\varepsilon}

\newcommand{\la}{\lambda}
\newcommand{\om}{\omega}
\newcommand{\si}{\sigma}
\newcommand{\te}{\theta}
\newcommand{\vp}{\varphi}

\newcommand{\Si}{\Sigma}

\def\RR{\mathbb{R}}

\def\ZZ{\mathbb{Z}}

\def\TT{\mathbb{T}}

\newcommand{\cA}{{\mathcal A}}
\newcommand{\cB}{{\mathcal B}}
\newcommand{\cC}{{\mathcal C}}
\newcommand{\cD}{{\mathcal D}}
\newcommand{\cE}{{\mathcal E}}

\newcommand{\cI}{{\mathcal I}}

\newcommand{\cM}{{\mathcal M}}
\newcommand{\tcM}{\widetilde{{\mathcal M}}}

\newcommand{\cO}{{\mathcal O}}

\newcommand{\cR}{{\mathcal R}}

\newcommand\cZ{\mathcal Z}

\newcommand{\pd}{\partial}
\newcommand\minus\backslash

\newcommand\lan\langle
\newcommand\ran\rangle

\newcommand{\tr}{\operatorname{tr}}

\newcommand{\un}{^{\mathrm u}}
\newcommand{\s}{^{\mathrm s}}
\DeclareMathOperator\Div{div} 
\DeclareMathOperator\Real{Re}

\renewcommand\leq\leqslant
\renewcommand\geq\geqslant
\newlength{\intwidth}

\DeclareMathOperator\Imag{Im}

 \DeclareMathOperator\curl{curl}

\begin{document}

\title[Stationary
  phase methods and the splitting of separatrices]{Stationary
  phase methods\\ and the splitting of separatrices}

\author{Alberto Enciso}
\address{Instituto de Ciencias Matem\'aticas, Consejo Superior de
  Investigaciones Cient\'\i ficas, 28049 Madrid, Spain}
\email{aenciso@icmat.es}

\author{Alejandro Luque}
\address{Department of Mathematics, Uppsala University, 751 06
  Uppsala, Sweden}
\email{alejandro.luque@math.uu.se}

\author{Daniel Peralta-Salas}
\address{Instituto de Ciencias Matem\'aticas, Consejo Superior de
 Investigaciones Cient\'\i ficas, 28049 Madrid, Spain}
\email{dperalta@icmat.es}

%%    General info
%\subjclass[2010]{35B38, 58J05, 58K45}
%\date{\today}
%
%\keywords{ }
%
\begin{abstract}
  Using stationary phase methods, we provide an explicit formula for
  the Melnikov function of the one and a half degrees of freedom
  system given by a Hamiltonian system subject to a rapidly
  oscillating perturbation. Remarkably, the Melnikov
  function turns out to be computable without an explicit knowledge
  of the separatrix and in the case of non-analytic systems.  This is related to a priori stable systems
  coupled with low regularity perturbations. It also applies to
  perturbations controlled by wave-type equations, so in particular we also
  illustrate this result with the motion of charged particles
  in a rapidly oscillating electromagnetic field. Quasi-periodic
  perturbations are discussed too.
\end{abstract}
\maketitle

%%%%%%%%%%%%%%%%%%%%%%%%%%%%%%%%%%%%%%%%
\section{Introduction}
%%%%%%%%%%%%%%%%%%%%%%%%%%%%%%%%%%%%%%%%

One of the most fundamental problems in celestial and Hamiltonian
mechanics is to ascertain whether a given system is chaotic. Even
though this question has been thoroughly studied for over a century,
the only existing method to address this question, which dates back to
Poincar\'e, is to consider perturbations of an explicit separatrix
of the system (more precisely, a homoclinic or
heteroclinic connection associated with a hyperbolic equilibrium) to produce transverse
intersections of a stable manifold and an unstable manifold. This is
well-known to imply the existence of a chaotic invariant set with
positive topological entropy~\cite{Smale}. Under suitable
technical hypotheses, the converse implication also holds in low dimension, so the existence of transverse homoclinic connections and positive
entropy are in fact equivalent~\cite{Katok}.

Again since Poincar\'e, the way to analyze the intersections of stable
and unstable manifolds is through the computation of the so called
{\em displacement function}\/, which measures the distance between these
manifolds. The definition of the displacement function will be
recalled in Section~\ref{S.proof}. For our present purposes it
suffices to keep in mind that when the perturbation can be thought of in terms of a
perturbation parameter, the leading order term of the displacement
function is usually called the {\em Melnikov function}\/, which is
given by an explicit integral that we will write down shortly. A serious
difficulty in mechanics is that there are systems, called {\em a~priori
stable}\/, whose displacement function is exponentially small in the
perturbation parameter, so one cannot define a nontrivial
Melnikov function and the problem is not amenable to a perturbative
analysis. It is well known that the appearance of a priori stable
systems is associated with hyperbolic equilibria whose eigenvalues
are small in the perturbation parameter or, equivalently modulo a
rescaling of the time variable, to perturbations that oscillate
rapidly in time.

To introduce the explicit expression of the Melnikov integral, let
us for concreteness consider the classical perturbative setting of a
one and a half degrees of freedom system of the form
\begin{equation}\label{eqx}
\ddot x= f(x)+\ep^r g(x,\dot x,t; \ep)\,,
\end{equation}
where $r>0$ and $x$ takes values in the real line. This is the kind of systems that one typically
gets in the study of perturbations of Hamiltonian systems with one degree of freedom
and in many reductions of a three-dimensional system. We
{can}
assume
that the unperturbed system ($\ep=0$) has a homoclinic trajectory
given by
a
stable manifold and
an unstable manifold of a hyperbolic
equilibrium, which one can take to be the origin
$(x,\dot x)=(0,0)$.
The homoclinic trajectory itself will be denoted by $x_0(t)$, and
the Melnikov integral $\cM(t_0)$ is then the one-variable function defined
as~\cite[Section 4.5]{GH}
\begin{equation}\label{Melnikov}
\cM(t_0):=\ep^r\int_{-\infty}^\infty \dot x_0(t)\, g(x_0(t),\dot
x_0(t),t+t_0;\ep)\, dt\,.
\end{equation}
When the function~$g$ and all its derivatives are locally bounded
uniformly in~$\ep$, it is standard that the Melnikov integral captures
the leading order of the displacement function~$\cD(t_0)$ in the sense
that $\cM(t_0)$ is of order~$\ep^r$ and
\[
  \cD(t_0)=\cM(t_0) + O(\ep^{r+1}) \,.
\]

When the perturbation~$g(x,\dot x,t; \ep)$ and all its derivatives are
well behaved at $\ep=0$, the evaluation of the Melnikov integral does
not present any conceptual difficulties (although, of course, explicit
formulas are often impossible to obtain).  However, when the
perturbation oscillates wildly as $\ep\to0$, the Melnikov function,
whose connection with the displacement is also subtler, can become
exponentially small and the analysis becomes much more
involved~\cite{Viejo2, Viejo3, Viejo5, Viejo6, Viejo9, Guardia1, Guardia3}. This is not surprising as this
highly oscillatory situation corresponds to the case of a priori
stable systems. In order to see this, consider the model a priori
stable problem of a small amplitude pendulum with a time-periodic
perturbation,
\[
\frac{d^2 x}{d\tau^2}= \ep^2\sin x +\ep^{r+2} G\bigg(x,\frac{dx}{d\tau}\bigg)\, \cos \tau\,,
\]
where $G(x,v)$
is a smooth function.
After rescaling the time variable as $t:=\ep\tau$, the system can be
written in the form~\eqref{eqx} with
\begin{equation}\label{stab}
f(x)=\sin x\,,\qquad g (x,\dot x,t; \ep):= G(x,\ep\dot x)\, \cos\frac t\ep\,,
\end{equation}
so in this particular example the $C^j$~norm of~$g$ is
\[
\|\ep^r g (\cdot,\cdot ,\cdot; \ep)\|_{C^j}=O(\ep^{r-j})
\]
for $\ep$ close to~0. More precisely, for all~$R$ one has
\begin{equation}\label{fastt}
C_1{\ep^{r-j+l}}  \leq\sup_{t\in\RR, \; |x|+ |\dot x|<R} \ep^r \big|\pd_t^j\pd_x^k\pd_{\dot x}^lg(x,\dot x,t;
\ep)\big|\leq C_2 \ep^{r-j+l}
\end{equation}
with a positive constant $C_j$ that depends on~$R$ but not on~$\ep$. For
this kind of perturbations, it is standard (see e.g ~\cite{Lombardi})
that the Melnikov function is exponentially small in~$\ep$.

The present paper is centered around the observation that, however, it
is possible to analyze the splitting of separatrices in an a priori
stable system in the perturbative regime provided that the
perturbation features fast oscillations in the space variables. This
is achieved by writing the Melnikov function as an oscillatory
integral that can be evaluated using a robust method based on
stationary phase arguments. Before explaining this mechanism in
detail, let us discuss briefly the meaning of these
perturbations. Basically,
{they}
correspond again to perturbations that
are small in a low norm (say, the uniform norm) but large in higher
norms, but with the additional twist that the norm gets larger not
only through time derivatives as in~\eqref{fastt} but also through
space derivatives. Therefore, these systems model small perturbations
of low regularity, in a quantitative way. These perturbations appear
naturally in systems controlled through a PDE (prime examples would be
non-smooth geometries generated via Nash--Moser
iterations~\cite{Nash,Lellis} or problems in fluid
mechanics~\cite{NS}), in systems with random Gaussian
perturbations~\cite{Sodin}, or in equations on manifolds defined
through diffeomorphisms that are close to the identity in the uniform
norm but not in higher norms. The latter arise e.g.\ in Eliashberg and
Mishachev's approach to Gromov's h-principle~\cite{Eliashberg}, which
is central to many questions in contact geometry, and even in the
study of magnetic fields generated by knotted wires~\cite{Ulam}. For
the benefit of the reader, a concrete physical example of a charged
particle moving in a rapidly oscillating electromagnetic field is
studied in detail in Section~\ref{S.ex}.

To explain how this kind of perturbation can split the separatrices of
an a priori stable system in the perturbative regime, for concreteness
let us consider perturbations of the form
\begin{equation}\label{g}
g (x,\dot x,t; \ep) = q\bigg(\frac x\ep,\dot x;\ep\bigg)\cos\frac t\ep \,,
\end{equation}
where the function $q(\xi,v;\ep)$ is a $C^1$~function and its derivatives are
locally bounded by $\ep$-independent
constants.
To keep things simple, we will assume that the function
$q(\xi,v;\ep)$ is $2\pi$-periodic in~$\xi$.
These perturbations, of course, do not
satisfy the bound~\eqref{fastt}, as otherwise the Melnikov function
would be exponentially small. It is worth mentioning that a
perturbation that does not satisfy that bound either was studied in~\cite{Guardia2}, where the authors took the analytic
perturbation of the form
\[
g(x,t;\ep):=\ep^r \frac{\sin x}{(1+(1-\ep^2)\sin x)^2} \cos \frac{t}\ep\,,
\]
which satisfies the uniform bounds
\[
C_1\ep^{r-2-4l}\leq \|g(\cdot,t;\ep)\|_{C^l} \leq C_2\ep^{r-2-4l}\,.
\]
Using residues, the authors manage to compute the leading order
contribution to the displacement function, which is not exponentially
small in~$\ep$. Using harmonic analysis one can see that a significant
part of the $L^2$~norm of the function is concentrated on frequencies
of order~$\ep^{-4}$, and a (straightforward but rather messy)
extension of the ideas discussed in this paper could be used to
analyze this kind of perturbation without any analyticity assumptions.

The main result of the paper is the following, which provides an
explicit formula for the leading order of the displacement function for
rapidly oscillating perturbations as above. We will state it in terms
of the Fourier coefficients~$q_k(v)$ of~$q(\xi,v;0)$, defined as
\[
q_k(v):=\frac1{2\pi}\int_0^{2\pi} q(\xi,v;0)\, e^{-ik\xi}\, d\xi\,.
\]
We will state the theorem
in the case of homoclinic connections, but obviously an analogous
result holds for heteroclinic connections:

\begin{theorem}\label{T.Melnikov}
Let us consider the system~\eqref{eqx} with $r>\frac52$ and a
perturbation of the form~\eqref{g}, where the functions~$f$ and~$g$
are of class~$C^3$. Suppose that the unperturbed system ($\ep=0$) has a hyperbolic equilibrium
  with a homoclinic connection corresponding to a stable manifold and
  an unstable manifold, which we describe through
  an integral curve $x_0(t)$. Setting
\[
\cC_k:=\{ t^*\in\RR: \dot x_0(t^*)=1/k\}\,,
\]
assume moreover that $|f(x_0(t^*))|\neq0$ for all $t^*\in \cC_k$ and
any nonzero integer~$k$.
Then the displacement function is
\[
\cD(t_0)=\ep^{r+\frac12}\Big(\cA\cos
\frac{t_0}\ep+\cB\sin\frac{t_0}\ep\Big) + O(\ep^{r+1}+\ep^{{2r-2}})\,,
\]
where $\cA$ and $\cB$ are the real constants
\begin{align*}
\cA&:= \sqrt{2\pi} \Real\left[\sum_{
k\in\ZZ\backslash\{0\}
} \sum_{t^*\in\cC_k} \frac{ q_k(k^{-1})\, e^{i(k x_0(t^*)-t^*+\sigma^* \pi/4)}}{k|k|^{1/2}|f(x_0(t^*))|^{1/2}}\right] \,,\\
\cB&:= \sqrt{2\pi} \Imag\left[\sum_{
k\in\ZZ\backslash\{0\}
} \sum_{t^*\in\cC_k} \frac{ q_k(k^{-1})\, e^{i(k x_0(t^*)-t^*+\sigma^* \pi/4)}}{k|k|^{1/2}|f(x_0(t^*))|^{1/2}}\right] \,,
\end{align*}
where $\si^*:= kf(x_0(t^*))/|kf(x_0(t^*))|$
is plus or minus
one, and the cardinality of~$\cC_k$ is uniformly bounded in~$k$.
\end{theorem}

This theorem presents three unusual features that should be carefully
noticed. Firstly, unlike essentially all the other results in the
literature, analyticity is not required: both the unperturbed system
and the perturbation can be of class~{$C^3$}. Secondly, to compute
this formula one does not need to know the separatrix explicitly,
which makes it quite versatile. Thirdly, to effectively apply this formula, one only needs to know that
  $\cA$ and~$\cB$ are not both zero. This condition is satisfied
  generically, and also computable, so in a concrete example one could
  even verify this condition by means of a computer assisted
  proof. The proof of this theorem is presented in Section~\ref{S.proof}.

Since the main application of the splitting of separatrices is to
establish that a certain dynamics is chaotic, it is worth stating the
following immediate corollary, which provides a computable generic
condition for the existence of positive topological entropy:

\begin{corollary}\label{C.splitting}
If the generic condition $\cA^2+\cB^2\neq0$ is satisfied, the perturbed system exhibits transverse
homoclinic intersections for all small enough~$\ep>0$. In particular,
it has positive topological entropy.
\end{corollary}

% \marginpar{Propuesa de remark. Para que no parezca que si
% $r\leq 5/2$ el Melnikov podria no dominar.}
% \begin{remark}
% \textcolor{blue}{
% The condition $r>5/2$ is required to ensure that the size of the splitting
% is well predicted by the Melnikov function, so that we can neglect the
% term $O(\ep^{2r-2})$. As we will point out in Remark~\ref{rem:rem}, this
% bound on the remainder is quite pessimistic and corresponds to a brute
% overestimating of an oscillatory integral. Thus, the condition $r>5/2$
% can be improved by repeating mutatis mutandis the
% computations presented in this paper.}
% \end{remark}

It is worth mentioning that the ideas of the proof of
Theorem~\ref{T.Melnikov} apply to much more general perturbations. As
a rule of thumb, one can handle
perturbations that oscillate roughly like~\eqref{g}, meaning that a
significant part of the perturbation (say, as measured with the $L^2$~norm) is
concentrated over frequencies of order~$1/\ep$. In
particular, in Section~\ref{S.qp} we will consider perturbations that
are quasiperiodic in~$t$, say of the form
\[
g (x,\dot x,t; \ep) = q\bigg(\frac x\ep,\dot x;\ep\bigg)\,
F\Big( \frac{\om_1 t}\ep,\dots, \frac{\om_n t}\ep\Big)
\]
with a~$C^3$ function $F:\TT^n\to\RR$, and derive an analogous expression for the displacement
function (Theorem~\ref{T.qp}). To keep the technicalities to a minimum, we have chosen not
to state the result in the Introduction.

% \section{A stationary phase lemma}
% \label{S.lemma}

\section{Proof of Theorem~\ref{T.Melnikov}}
\label{S.proof}

Without any loss of generality we can assume that the hyperbolic
equilibrium of the unperturbed system
\[
\ddot x= f(x)
\]
is located at the origin $z=0$, with
\[
z:=(x,\dot x)\,.
\]
Likewise, it is convenient to introduce the notation
\[
z_0(t):=(x_0(t),\dot x_0(t))
\]
for the integral curve parametrizing the homoclinic connection. Since
the function~$f$ is of class~$C^3$, $z_0$ is a $C^3$~function of time.

\subsubsection*{Step 1: The Melnikov function and estimate for the
  error term}

Let us fix a point in the separatrix (understood as a curve in the
two-dimensional phase space of coordinates $z=(x,\dot x)$), for example $Z:=z_0(0)$, and
consider a normal section to the separatrix at that point, which is
given by a segment~$\Si$ passing through the point~$Z$ and parallel to
the vector $ (-f(x_0(0)), \dot x_0(0))$.
In the rest of the proof, $K$~will denote a compact subset of the
plane containing the homoclinic connection. Since the perturbation is $C^1$-bounded as
\begin{equation}\label{C1g}
  \|\ep^r g(\cdot,\cdot,\cdot;\ep)\|_{C^1(K\times\RR)}< C\ep^{r-1}
\end{equation}
because~$g$ is of the form~\eqref{g}, the
invariant manifold theorem~\cite{HPS} ensures that for small
enough~$\ep$ the perturbed system also has a hyperbolic
point~$p_{\ep,t_0}$
close to the origin, and that it has a stable manifold and an unstable
manifold locally close to the separatrix. Here, with some abuse of
notation, we are saying that~$p_{\ep,t_0}$ is a hyperbolic point in
the sense that it is a fixed point of the return map corresponding to
the section $\{t=t_0\}$ for the extended system $\ddot x= f(x)+\ep^r\,
g(x,\dot x,t;\ep)$, $\dot t=1$.

More precisely, the first
intersection points~$Z_{\ep,t_0}\s$, $Z_{\ep,t_0}\un$ of the stable
and unstable manifolds of the perturbed system with the segment~$\Si$
are at a distance at most $C\ep^{r-1}$ of~$Z$. Furthermore, the integral
curves
\[
z_\ep\s(t;t_0)= (x_\ep\s(t;t_0),\dot x_\ep\s(t;t_0))\, ,\qquad  z_\ep\un(t;t_0)= (x_\ep\un(t;t_0),\dot x_\ep\un(t;t_0))
\]
of the perturbed system, defined through the initial conditions
\[
z_\ep\s(t_0;t_0)=Z_{\ep,t_0}\s\,,\qquad z_\ep\un(t_0;t_0)=Z_{\ep,t_0}\un\,,
\]
satisfy
\begin{equation}\label{boundsz}
\sup_{t>t_0} |z_\ep\s(t;t_0)-z_0(t-t_0)|+ \sup_{t<t_0} |z_\ep\un(t;t_0)-z_0(t-t_0)|<C\ep^{r-1}\,.
\end{equation}
In the case of the unperturbed system, notice that the eigenvalues of the hyperbolic equilibrium
are of the form $\pm\la$ with $\la>0$ a constant independent
of~$\ep$, so it is well known that the unperturbed integral curve
satisfies
\begin{equation}\label{asymptz0}
|x_0(t)|+ |\dot x_0(t)| + |\ddot x_0(t)|<Ce^{-\la|t|}
\end{equation}
for all~$t$.
% Notice, moreover, that the eigenvalues of the hyperbolic equilibrium
% of the perturbed system are of the form
% \[
% \la_{\ep,t_0}\un=\la + O(\ep^{r-1})\,,\qquad \la_{\ep,t_0}\s=-\la + O(\ep^{r-1})\,,
% \]
% with $\la>0$ an $\ep$-independent constant, so in particular the above
% integral curves tends exponentially to the equilibrium point as
% \begin{equation}\label{asymptz}
% |z_\ep\s(t;t_0)- (p_{\ep,t_0},0)|+ |z_\ep\s(-t;t_0)- (p_{\ep,t_0},0)| < C e^{-[\la+O(\ep^{r-1})]t}
% \end{equation}
% for $t>0$.
% Likewise, there is an easy $L^1$~version of~\eqref{boundsz}
% asserting that
% \begin{equation}\label{boundsz2}
% \int_{t_0}^\infty |\dot x_\ep\s(t;t_0)-\dot x_0(t-t_0)|dt+ \int_{-\infty}^{t_0}|\dot x_\ep\un(t;t_0)-\dot x_0(t-t_0)|dt<C\ep^{r-1}\,.
% \end{equation}
% Obviously in the case of the unperturbed system one has
% analogous estimates, such as $|z_0(t)|<Ce^{-\la|t|}$ for all~$t$.

It is standard~\cite[Section 4.5]{GH} that the displacement function is defined as
\begin{equation}\label{displacement}
\cD(t_0):= \al\big[z_\ep\un(t_0;t_0)-z_\ep\s(t_0;t_0)\big]\,,
\end{equation}
where for later convenience we have introduced an inessential nonzero constant
independent of the perturbation parameter,
\[
\al:=\big[\dot x_0(0)^2+f(x_0(0))^2\big]^{\frac12}\,.
\]
Since the perturbations that we are considering are rapidly
oscillating, the connection between the Melnikov integral and the
displacement function is rather different from the usual one:

\begin{proposition}\label{P.Melnikov}
The Melnikov function determines the leading order contribution to the
displacement function in the sense that
\[
|\cD(t_0)-\cM(t_0)|\leq C\ep^{2r-2}
\]
uniformly on compact time intervals.
\end{proposition}
\begin{proof}
To derive the formula, let us write the perturbed equation as a first order differential equation
\[
\dot z = X_f(z) + \ep^r X_g(z,t;\ep)\,,
\]
where we are using the notation
\[
X_f(z):=(\dot x, f(x))\,,\qquad X_g(z,t;\ep):=(0,g(x,\dot x,
t,\ep))\,.
\]
It is clear that
\[
\cD(t_0)= \cD\un(t_0;t_0)-\cD\s(t_0;t_0)\,,
\]
where the time-dependent displacement functions are defined as
\begin{align*}
\cD\un (t;t_0)&:= X_f(z_0(t-t_0)) \times (z_\ep\un(t;t_0)-z_0(t-t_0))\,,\\
\cD\s (t;t_0)&:=X_f(z_0(t-t_0)) \times (z_\ep\s(t;t_0)-z_0(t-t_0))
\end{align*}
and we have defined the cross product of two vectors in~$\RR^2$ as the
scalar quantity
\[
X\times Y:= X_1Y_2-X_2Y_1\,.
\]

Taking the case of the unstable manifold for concreteness, let us
compute the derivative of $\cD\un$ with respect to~$t$. To keep the
expressions simple, we will omit the arguments of the functions
when no confusion may arise:
\begin{align*}
\dot \cD\un = {} & DX_f (z_0) \dot z_0 \times (z_\ep\un -z_0) + X_f(z_0) \times (\dot z_\ep\un-\dot z_0) \\
= {} & DX_f (z_0) \dot z_0 \times (z_\ep\un -z_0) + X_f(z_0) \times [X_f(z_\ep\un)+\ep^r X_g(z_\ep\un) - X_f(z_0)] \\
= {} & DX_f (z_0) \dot z_0 \times (z_\ep\un -z_0) + X_f(z_0) \times [DX_f(z_0) (z_\ep\un -z_0)+\ep^r X_g(z_0)] + \cR \\
= {} & (\tr DX_f)\, \dot z_0 \times (z_\ep\un -z_0) + \ep^r X_f(z_0) \times X_g(z_0) + \cR\\
= {} & \ep^r \dot x_0 (t-t_0)\, g(x_0(t-t_0),\dot x_0(t-t_0), t;\ep) + \cR\,,
\end{align*}
where the trace of~$DX_f$ (which is zero) appears because
\[
 (\tr DX_f)\, \dot z_0 \times (z_\ep\un -z_0) =  DX_f (z_0) \dot z_0 \times (z_\ep\un -z_0) + X_f(z_0) \times [DX_f(z_0) (z_\ep\un -z_0)]
\]
and  the error term,
\begin{multline*}
\cR:= X_f(z_0)\times \big[ X_f(z_\ep\un)-X_f(z_0)-DX_f(z_0)\,
       (z_\ep\un-z_0)\big] \\
+ \ep^rX_f\times
  [X_g (z_\ep\un)-X_g(z_0)]\,,
\end{multline*}
is bounded for $t\leq t_0$ as
\begin{align*}
|\cR|
& \leq C\|X_f\|_{C^2(K\times\RR)}|X_f||z_\ep\un -z_0|^2+ C\ep^r
  \|X_g\|_{C^1(K\times\RR)}|X_f||z_\ep\un-z_0|\\
& \leq C\ep^{2r-2}|X_f|\,,
\end{align*}
where we have used the bounds~\eqref{boundsz}. The stable displacement
function can
be estimated for $t\geq t_0$ using a completely analogous argument. Since $f(0)=0$, using the bound~\eqref{asymptz0} together with the obvious fact that
\[
\lim_{t\to\infty} \cD\s(t;t_0)=\lim_{t\to-\infty}\cD\un(t;t_0)=0\,,
\]
we can integrate the above expression for
the derivative of the stable and unstable parts of the displacement
function to obtain
\[
|\cD(t_0)- \cM(t_0)|\leq C\ep^{2r-2}\,,
\]
as claimed.
\end{proof}

\begin{remark}\label{rem:rem}
In what follows we will only need a uniform bound for the
difference~$\cD(t_0)-\cM(t_0)$, which is what we prove. The method of
proof, however, permits to estimate higher order derivatives too
if~$f$ and~$g$ are regular enough, and one should notice that they are
{\em not}\/ of order~$\ep^{2r-2}$, even if $f$ and~$g$ are analytic, due to the rapid oscillations of the
perturbation.
\end{remark}

\subsubsection*{Step 2: Estimates for the critical points of the phase
functions}

Let us now write the Melnikov integral as
\begin{align}
\cM(t_0)&= \ep^r\int_{-\infty}^\infty \dot x_0(t)\,
          q\bigg(\frac{x_0(t)}{\ep},\dot x_0(t);\ep\bigg)\,
          \cos\frac {t+t_0}\ep\, dt\notag\\
%&= \ep^r\int_{-\infty}^\infty \dot x_0(t)\,
%          q\bigg(\frac{x_0(t)}{\ep},\dot x_0(t);0\bigg)\,
%          \cos\frac {t+t_0}\ep\, dt+\cR_1(t_0)\notag\\
&= \tcM(t_0)+\cR_1(t_0)\,, \label{error1}
\end{align}
where
\[
\tcM(t_0):= \ep^r\int_{-\infty}^\infty \dot x_0(t)\,
          q\bigg(\frac{x_0(t)}{\ep},\dot x_0(t);0\bigg)\,
          \cos\frac {t+t_0}\ep\, dt
\]
and the error
\begin{align*}
\cR_1(t_0)&:= \ep^r\int_{-\infty}^\infty \dot x_0(t)\,
          \bigg[q\bigg(\frac{x_0(t)}{\ep},\dot x_0(t);\ep\bigg) -q\bigg(\frac{x_0(t)}{\ep},\dot x_0(t);0\bigg)\bigg]\,
          \cos\frac {t+t_0}\ep\, dt
\end{align*}
is uniformly bounded as
\begin{equation}\label{cR1}
|\cR_1(t_0)|\leq
\ep^r\|q(\cdot,\cdot;\ep)-q(\cdot,\cdot;0)\|_{C^0(K)}\int_{-\infty}^\infty
|\dot x_0(t)|\, dt\leq C\ep^{r+1}
\end{equation}
because
{$q(\xi,v;\ep)$}
depends differentiably on~$\ep$ and $|\dot
x_0(t)|$ falls off exponentially by~\eqref{asymptz0}.

Writing $\tcM(t_0)$ in terms of the Fourier
coefficients $q_k(v)$ of $q(\xi,v;0)$,
we have
\[
\tcM(t_0):= \ep^r\sum_{k=-\infty}^\infty \int_{-\infty}^\infty \dot x_0(t)\,
          q_k(\dot x_0(t))\,e^{ikx_0(t)/\ep}
          \cos\frac {t+t_0}\ep\, dt\,.
\]
Setting
\[
\vp_k(t):= kx_0(t)-t\,,
\]
this yields
\begin{align}
\tcM(t_0)&= \ep^r\sum_{k=-\infty}^\infty \Real \bigg[ e^{-it_0{/\ep}}\int_{-\infty}^\infty \dot x_0(t)\,
          q_k(\dot x_0(t))\,e^{i\vp_k(t)/\ep}\,
           dt\bigg]\notag\\
&=:\ep^r\sum_{k=-\infty}^\infty \Real [e^{-it_0/\ep}m_k]\,. \label{mks}
\end{align}

Our goal is to compute the constants~$m_k$ to leading order
in~$\ep$. Notice that the set~$\cC_k$ defined in the statement of
Theorem~\ref{T.Melnikov} is the locus of the critical points of the
phase function~$\vp_k$. We will need the following estimates for the
behavior of the phase function~$\vp_k$ around points in the
set~$\cC_k$:

\begin{lemma}\label{L.cCk}
There is a number~$N$, independent of~$k$ and~$\ep$, such that $\cC_k$
consists of at most~$N$ points~$t_{k,j}^*$. Moreover, there are
positive constants~$\eta,c,C$ and $T_0$, independent of~$k$ and~$\ep$, such that the
function
\begin{equation}\label{vpk}
\vp_k(t):=k x_0(t)-t
\end{equation}
satisfies
\begin{equation}\label{unif}
|\dot\vp_k(t)|>c
\end{equation}
for any~$k$ provided that $t$~does not lie on any of the
pairwise disjoint intervals $I_{k,j}:=(t_{k,j}^*-2\eta,
t_{k,j}^*+2\eta)$. Furthermore, for all~$t\in I_{k,j}$ one has:
\begin{enumerate}
\item If $|t^*_{k,j}|\leq T_0$,
\[
|\dot\vp_k(t)|\geq c|k||t-t^*_{k,j}|\,, \qquad |\ddot\vp_k(t^*_{k,j})|>c|k|\,,\qquad|\ddot\vp_k(t)| + |\dddot\vp_k(t)|<C|k|\,.
\]
\item If $|t^*_{k,j}|>T_0$,
\[
|\dot\vp_k(t)|\geq c|t-t^*_{k,j}|\,,\qquad
|\ddot\vp_k(t^*_{k,j})|>c\,,\qquad |\ddot\vp_k(t)| + |\dddot\vp_k(t)|<C\,.
\]
\end{enumerate}
\end{lemma}

\begin{proof}
Let us start by noticing that $\cC_0$ is
the empty set, so in what follows we will assume that
$k\neq0$.
A first observation is that~$\cC_k$ is a discrete set. This follows
from the fact that the
zeros of the function
\begin{equation}\label{zeros}
\dot x_0(t^*)-\frac1k
\end{equation}
(which are the points in~$\cC_k$) are non-degenerate, as its derivative
\begin{equation}\label{der}
\ddot x_0(t^*)=f(x_0(t^*))
\end{equation}
is nonzero whenever $t^*\in\cC_k$ by hypothesis.
Hence, there is a finite number of zeros of Equation~\eqref{zeros}
on any compact interval, for any given~$k$. We will show next that the
bound is uniform in~$k$ by considering the case $|k|\to\infty$.

The key is to show that the solutions to the limit equation
\begin{equation}\label{limit}
\dot x_0(t)=0
\end{equation}
are finite. Notice that the derivative of this function, also given
by~\eqref{der}, is nonzero on the solutions of~\eqref{limit} because
there cannot be any equilibria on the homoclinic connection (recall
that this simply means that
\[
(\dot x_0(t),f(x_0(t)))\neq (0,0)
\]
for all $t\in\RR$). This implies that the set of solutions
of~\eqref{limit}, which we call
$\{T^*_m\}$ is finite on any compact
subset of the real line.
Restricting our attention to a compact interval $|t^*|<T$, it is then
a straightforward consequence of the implicit function theorem and the
analysis of the limit equation~\eqref{limit} that the number of zeros
of~\eqref{zeros} with $|t^*|<T$ remains bounded as $|k|\to\infty$.
Notice that it follows from this argument that the intersection of the
set~$\cC_k$ with any large finite interval $[-T,T]$ (with $T$
independent of~$k$ and~$\ep$) converges as $|k|\to\infty$ to the set of solutions of the
limit equation~\eqref{limit}, so in particular the cardinality of
$\cC_k\cap [-T,T]$ is bounded uniformly in~$k$. Notice moreover that, since
\begin{gather*}
\dot\vp_k(t)=k\dot x_0(t)-1\,,\quad \ddot\vp_k(t)=kf(x_0(t))\,,\quad
\dddot\vp_k(t)=kf'(x_0(t))\,\dot x_0(t)\,,
\end{gather*}
one immediately infers the bounds presented in item~(i).

Hence we can concentrate on controlling what happens for large~$|t|$.
Specifically, we will show that the number of zeros
of~\eqref{zeros} with $t^*>T$ or $t^*<-T$ is at most one, for large enough~$T$.
To this end, let us recall that, since $x_0(t)$ is a homoclinic connection of a hyperbolic equilibrium,
it is standard that there
are nonzero constants~$c_+$, $c_-$ such that
\begin{equation}\label{asympt}
\dot x_0(t)= \begin{cases}
\la  c_+ \, e^{-\la t}\, (1+\cE_+(t)) &\text{for } t>0\,,\\
\la  c_-  \, e^{\la t}\, (1+\cE_-(t)) &\text{for } t<0\,,
\end{cases}
\end{equation}
where the errors~$\cE_\pm(t)$ are bounded as
\begin{equation}\label{boundcE}
|\cE_\pm(t)|+|\dot\cE_\pm(t)|
+|\ddot\cE_\pm(t)|
< Ce^{-\la |t|}
\end{equation}
as $t\to\pm\infty$. In order to see this, observe that, by the stable
manifold theorem and the fact that $f$ is of class $C^3$,
the unstable manifold contained in the homoclinic
connection of the unperturbed system can be locally written as a
${C^3}$~graph  in the $(x,\dot x)$-plane,
that is,
\[
\dot x = h(x)
\]
with $h$ a ${C^3}$~function defined in a neighborhood of~0 with the
asymptotic behavior
\[
h(x)=\la x+ O(x^2)\,.
\]
As a hyperbolic one-dimensional system of class~$C^3$ can be conjugated to its linear
part through a $C^3$~diffeomorphism, there is a local variable~\cite{Beli}
\[
y:=\Phi(x)
\]
in a neighborhood of~0, with
\[
\Phi\in C^3\,,\qquad \Phi(0)=0\,,\qquad \Phi'(0)=1\,,
\]
such that the integral curve describing the unstable manifold can be
written in this variable as
\[
y(t)=c_-e^{\la t}\,.
\]
Inverting the diffeomorphism to write the unstable manifold in the
original variable~$x$, one readily obtains that for $t\to-\infty$ it is parametrized as
\[
x(t)=c_-e^{\la t}[1+\cO(e^{\la t})]\,,
\]
where the small term can be differentiated three times with the same
bound. This proves the bound~\eqref{boundcE} for~$\cE_-(t)$. The argument for the stable
manifold is analogous.

An immediate consequence of the estimate~\eqref{boundcE} is that $\dot
x_0(t)\neq0$ for all large enough~$|t|$. Therefore, we infer that the
number of solutions of Equation~\eqref{limit}
is finite. To control the solutions with large $t^*$, let us rewrite the equation for the zeros of~\eqref{zeros} as
\begin{equation}\label{eqkk}
c_+ \la\, e^{-\la t^*}\, (1+\cE_+(t^*))=\frac1k\,.
\end{equation}
Here we are supposing that $k$ is a large enough number of the same
sign as~$c_+$, since otherwise there cannot be any solutions with
$t^*$ larger than some fixed constant~$T$ (independent of~$\ep$
and~$k$) by the bounds on~$\cE_+(t^*)$.

If one takes out the
error~$\cE_+(t^*)$, it is clear that the equation
\begin{equation}\label{eqkk2}
c_+\la\, e^{-\la t^*}=\frac1k
\end{equation}
then has a unique solution with $t^*>T$, given in terms of
a logarithm. A judicious application of the inverse function theorem
then shows that~\eqref{eqkk}
has exactly one solution with
$t^*>T$ too.
To see this, define a new variable $\tau:=e^{-\la t^*}$
and let $E_+(\tau)$ be the expression in this variable of the function
$ e^{-\la t^*} \cE_+(t^*).$
The bounds~\eqref{boundcE} ensure that
\[
|E_+(\tau)|\leq C\tau^2\,,\qquad  |E_+'(\tau)|\leq C\tau
\]
for small positive~$\tau$. As Equation~\eqref{eqkk} can be written as
\[
c_+[\la\tau + E_+(\tau)] =\frac1k\,,
\]
the existence of exactly one solution in the interval follows from our
knowledge of Equation~\eqref{eqkk2}. The analysis for $t^*<-T$ is
completely analogous. Notice that, as a byproduct of the above
argument, if $t^*_{k,j}>T_0$ one infers that
\begin{equation*}%\label{t*kj}
t^*_{k,j}= \frac{1+e_{k,j}^*}\la \log (c_+\la k)\,,\qquad
|e_{k,j}^*|<\frac C{\log|k|}\,,
\end{equation*}
which tends to infinity for large~$|k|$, and
\begin{equation}\label{x0t*}
x_0(t^*_{k,j})=-\frac{1+e_{k,j}}{\la k}\,,\qquad
|e_{k,j}|<\frac C{|k|}\,,
\end{equation}
which tends to the equilibrium, located at~0. Of course, the preceding
argument also shows that, if $|t-t^*_{k,j}|<c_0$ with $c_0$ a small
uniform constant, then
\begin{equation}\label{x0t}
\frac1{C|k|} < |x_0(t)|+|\dot x_0(t)|+|\ddot x_0(t)|+|\dddot x_0(t)|< \frac C{|k|}\,.
\end{equation}
We record this fact here for later use.

The previous analysis shows that the cardinality of~$\cC_k$ is
uniformly bounded and that the distance between the points
$t_{k,j}^*\in\cC_k$ is larger than some positive uniform
constant.
Let us now prove the uniform
bound~\eqref{unif}. Since $f(x_0(t^*_{k,j}))\neq0$ by hypothesis, it suffices to obtain
uniform bounds when $|k|\to\infty$ or $|t^*|\to\infty$. Of course, if
$|k|$ is large but $t^*$ remains in a compact interval, it follows from
the above comparison argument with the limit equation~\eqref{limit}
that the uniform bound~\eqref{unif} holds, so we can focus on the case of
large~$|t^*|$. Assuming for concreteness that $t>T$, with~$T$ a large
positive constant, one can now use the asymptotic formula~\eqref{asympt}
to write
\[
k\dot x_0(t)-1=k \la  c_+ \, e^{-\la t}\, (1+\cE_+(t))-1\,.
\]
If $t=t^*_{k,j}+s$, using that
\begin{equation}\label{equnif}
k \la  c_+ \, e^{-\la t^*_{k,j}}\, (1+\cE_+(t^*_{k,j}))=1
\end{equation}
it is possibly to simplify the above expression to obtain
\begin{equation*}
k\dot x_0(t)-1=e^{-\la s}\frac{ (1+\cE_+(t^*_{k,j}+s))}{(1+\cE_+(t^*_{k,j}))}-1\,.
\end{equation*}
In view of the bound~\eqref{boundcE}, it is now clear that
\[
|\dot\vp_k(t)|=|k\dot x_0(t)-1|>c
\]
if $|s|$ is larger than some uniform constant, say $2\eta$.
This proves the estimate~\eqref{unif}.

Finally, we differentiate the
asymptotic expression~\eqref{asympt} and employ the
formula~\eqref{equnif} to obtain
\[
\ddot\vp_k(t^*_{k,j})=k\ddot x_0(t^*_{k,j})=-\la + O(e^{-\la t^*_{k,j}})\,,
\]
and thus the desired lower bounds
\[
|\ddot\vp_k(t^*_{k,j})|>c \quad \text{and} \quad |\dot\vp_k(t)|>c|t-t^*_{k,j}|
\]
for all $t\in I_{k,j}$. Furthermore, the above asymptotic expression also ensures
that
\[
|\dddot\vp_k(t)|\leq C
\]
for all~$t\in I_{k,j}$, which implies the estimates presented in
item~(ii). The lemma then follows.
\end{proof}

\subsubsection*{Step 3: The oscillatory integrals of the Melnikov function}

Let us consider the intervals $I_{k,j}$ introduced in
Lemma~\ref{L.cCk} and take a smooth compactly supported
function~$\chi_{k,j}(t)$ that
{is equal to~1 on the interval $[t^*_{k,j}- \eta,t^*_{k,j}+ \eta]$}
and is supported in the larger interval~$I_{k,j}$.
Since~$\eta$
does not depend on~$\ep$ and~$k$, the
functions~$\chi_{k,j}$ and their derivatives can be assumed to be
bounded by constants independent of~$k$ and~$\ep$. We will also set
\[
\chi_{k,0}(t):=1-\sum_{j=1}^{N_k}\chi_{k,j}(t)
\,,
\]
{where $N_k$ is the cardinality of $\cC_k$.

Let us now write the integral~$m_k$ {in \eqref{mks}} as
\[
m_k= \sum_{j=0}^{N_k} m_{k,j}\,,
\]
where
\[
m_{k,j}:=\int_{-\infty}^\infty \dot x_0(t)\,
          q_k(\dot x_0(t))\,e^{i\vp_k(t)/\ep}\, \chi_{k,j}(t)\,
           dt
\]
Our first goal is to estimate the terms~$m_0$ and~$m_{k,0}$, where the phase function
does not have any critical points in the region where the integrand is
nonzero.
In the proof of these estimates we
will use that~\cite[Theorem 3.3.9]{Grafakos}, as the function $q(\xi,v;0)$
is of class $C^3(K)$, its Fourier coefficients decay as
\begin{equation}\label{decayqk}
\|q_k\|_{C^l(K_2)}\leq CM(1+|k|)^{l-3}
\end{equation}
for all~$0\leq l\leq3$, where
$$
M:=\|q\|_{C^3(K)}
$$
and the constant~$C$ does not depend on~$q$.
Here and in what follows, $K_l\subset\RR$ denotes the
projection of the set~$K\subset\RR^2$ on the $l$th coordinate, with $l=1,2$.

\begin{lemma}\label{L.mk0}
The terms $m_0$ and $m_{k,0}$ with $k\neq0$ are bounded as
\[
|m_0|<CM \ep\,,\qquad |m_{k,0}|<CM\ep|k|^{-2}\,.
\]
\end{lemma}
\begin{proof}
Let us begin with~$m_0$. As $\vp_0(t)=-t$ and $\ddot x_0=f(x_0)$, one can integrate by parts to show
\begin{align*}
|m_0|&= \bigg| \int_{-\infty}^\infty \dot x_0(t) \, q_0(\dot x_0(t))\,
       e^{-it/\ep}\, dt\bigg|\\
& =\ep\bigg| \int_{-\infty}^\infty e^{-it/\ep} \frac d{dt}\big( \dot
  x_0(t)\, q_0(\dot x_0(t))\big)\, dt\bigg|\\
&\leq \ep\int_{-\infty}^\infty |f(x_0(t))|\,\big({|} q_0(\dot x_0(t))|+ |\dot
  x_0(t)\, q_0'(\dot x_0(t))|\big)\, dt\\
&\leq CM\ep\,.
\end{align*}

Let us pass now to $m_{k,0}$ with $k\neq0$. Since
{$|\dot \vp_k(t)|>c$}
on the support
of~$\chi_{k,0}$ by Lemma~\ref{L.cCk}, one can use the identity
\[
e^{i\vp_k(t)/\ep}= -\frac{i\ep}{\dot\vp_k(t)} \frac d{dt} e^{i\vp_k(t)/\ep}
\]
and the exponential fall off~\eqref{asymptz0} to integrate by
parts, which permits to write the integral for~$m_{k,0}$ as
\begin{align*}
m_{k,0}= -i\ep \int_{-\infty}^\infty \frac d{dt}\bigg( \frac{\dot
  x_0(t)\, q_k({\dot x_0(t)} )\, \chi_{k,0}(t)}{\dot\vp_k(t)}\bigg)\,
  e^{i\vp_k(t)/\ep}\, dt\,.
\end{align*}
The derivative under the integral sign is
\[
\frac{\ddot x_0q_k\chi_{k,0}+{\dot x_0 \ddot x_0}
q_k'\chi_{k,0}+\dot
  x_0q_k\dot\chi_{k,0}}{\dot\vp_k} - \frac{\ddot\vp_k\dot x_0q_k\chi_{k,0}}{\dot\vp_k^2}\,,
\]
so we get the obvious bound
\begin{align*}
|m_{k,0}|\leq {} &
C\ep \frac{ \|q_k\|_{C^1(K_2)} \|\chi_{k,0}\|_{C^1(\RR)}}{c}
\int_{-\infty}^{\infty} (|\ddot x_0|+{|\dot x_0 \ddot x_0|}+|\dot x_0|)\, dt \\
& +
C\ep \frac{ |k|\|q_k\|_{C^0(K_2)} \|\chi_{k,0}\|_{C^0(\RR)}}{c^2}
\int_{-\infty}^{\infty} |\dot x_0|\, dt
\end{align*}
where we have employed that
$$
|\ddot\vp_k|=|k\ddot x_0|=|k f(x_0)|<C|k|
$$
and
$|\dot\vp_k|>c$ by~\eqref{unif}. Using now
the bound~\eqref{decayqk} for the decay of the Fourier coefficients, the lemma
follows.
\end{proof}

Let us now pass to estimate the numbers~$m_{k,j}$:

\begin{lemma}\label{L.mkj}
For all $j\in\cC_k$,
\[
m_{k,j}= \bigg(\frac{2\pi \ep}{|f(x_0(t^*_{k,j}))|}\bigg)^{\frac12} \frac{q_k(k^{-1})}{k|k|^{1/2}}
  e^{ikx_0(t^*_{k,j})-it^*_{k,j}+i\si_{k,j}\frac\pi
    4}+ \cR_{k,j}\,,
\]
where
\[
\si_{k,j}:=\frac{kf(x_0(t^*_{k,j}))}{|kf(x_0(t^*_{k,j}))|}
\]
is plus or minus one and
the error is bounded as
\[
|\cR_{k,j}|< CM\ep|k|^{-5/2}\,.
\]
\end{lemma}

\begin{proof}
The proof is a little different depending on whether~$|t^*_{k,j}|$ is larger
than the uniform constant~$T_0$ introduced in Lemma~\ref{L.cCk} or
not:

\subsubsection*{First case: $|t^*_{k,j}|\leq T_0$.} Let us consider a
function~$\chi^1_{k,j}(t)$ that is equal to~1 if
$|t-t^*_{k,j}|<\frac12c_0|k|^{-1/2}$, vanishes identically for
$|t-t^*_{k,j}|>c_0|k|^{-1/2}$, and is bounded as
\begin{equation}\label{boundschi1}
\sup_{t\in\RR}\Big(|\chi_{k,j}^1(t)|+|k|^{-\frac12}|\dot\chi^1_{k,j}(t)|
+ |k|^{-1} |\ddot\chi_{k,j}^1(t)|\Big)<C\,,
\end{equation}
where $c_0<\eta$ and $C$ are uniform constants. We will also set
\[
\chi^2_{k,j}(t):=\chi_{k,j}(t)-\chi_{k,j}^1(t)\,.
\]
This allows us to write
\[
m_{k,j}=m_{k,j}^1+m_{k,j}^2\,,
\]
where we have set
\begin{align*}
m_{k,j}^l:=\int_{-\infty}^\infty \dot x_0(t)\,
          q_k(\dot x_0(t))\,e^{i\vp_k(t)/\ep}\, \chi_{k,j}^l(t)\,
           dt
\end{align*}
with $l=1,2$. Using an integration by parts
argument similar to that of Lemma~\ref{L.mk0}, one can readily write
\begin{align*}
m_{k,j}^2&=-i\ep \int_{-\infty}^\infty \frac{\dot
  x_0\chi_{k,j}^2q_k({\dot x_0(t)} )}{\dot\vp_k}\frac d{dt} e^{i\vp_k/\ep}\,
           dt\\
&=i\ep \int_{-\infty}^\infty \frac d{dt} \bigg( \frac{\dot
  x_0\chi_{k,j}^2q_k({\dot x_0(t)})}{\dot\vp_k}\bigg)\, e^{i\vp_k/\ep}\,
           dt\,.
\end{align*}
Hence one can utilize the estimates for~$\vp_k(t)$ presented in
{item~(i)} of Lemma~\ref{L.cCk} to bound this term as
\begin{align*}
|m_{k,j}^2|&\leq C\ep
             \big[\|q_k\|_{C^0(K_2)}|k|^{-\frac12}+\|q_k\|_{C^1(K_2)}|k|^{-1}\big]\int_{\frac{c_0}2|k|^{-\frac12}<|s|<{2}\eta}\frac{ds}s
  \\
&\qquad \qquad \qquad \qquad +C\ep
  \|q_k\|_{C^0(K_2)}|k|^{-1}\int_{\frac{c_0}2|k|^{-\frac12}<|s|<{2}\eta}\frac{ds}{s^2}\\
&\leq C\ep [\|q_k\|_{C^0(K_2)}|k|^{\frac12}+\|q_k\|_{C^1(K_2)}]\frac{\log(1+|k|)}{|k|}\,.
\end{align*}
To see why these integrals appear in the estimate, notice that the way
one bounds each of the terms in the above integral using
Lemma~\ref{L.cCk} (together with the bounds \eqref{boundschi1} and~\eqref{asymptz0}) is as follows:
\begin{align*}
\bigg|\int_{-\infty}^\infty \frac{\dot x_0\dot\chi_{k,j}^2
  q_k}{\dot\vp_k}\,  e^{i\vp_k/\ep}\,
           dt\bigg| &\leq
  C\|q_k\|_{C^0(K_2)} \|\dot\chi_{k,j}^2\|_{C^0}{|k|^{-1}} \int_{\frac{c_0}2 {|k|^{-\frac12}}<|t-t^*_{k,j}|<2\eta}
                     \frac{dt}{|t-t^*_{k,j}|}\\
&\leq C\|q_k\|_{C^0(K_2)}{|k|^{-\frac12}}\int_{\frac{c_0}2{|k|^{-\frac12}} <|s|<2\eta}\frac{ds}s\,.
\end{align*}
% where
% \[
% \cI:=\big\{t: |t-t^*_{k,j}|<c_0|k|^{-\frac12}\big\}\,.
% \]s
The bounds~\eqref{decayqk} for the decay of the Fourier coefficients
then ensure that
\[
|m_{k,j}^2|<CM \ep|k|^{-3}\log(1+|k|)\,,
\]

Let us now estimate~$m_{k,j}^1$. Using Taylor's formula
at~$t^*_{k,j}$ and the fact that
\[
f_{k,j}:=f(x_0(t^*_{k,j}))
\]
satisfies
{$|f_{k,j}|>c$}
by Lemma~\ref{L.cCk}
{(item~(i))},
one can write
\[
\vp_k(t)=\al_{k,j}+ h(t)\, (t-t^*_{k,j})^2\,,
\]
where $\al_{k,j}:= \vp_k(t^*_{k,j})$,  $h(t^*_{k,j})=\frac12kf_{k,j}$ and
\begin{align*}
\frac{|h(t)-\frac12kf_{k,j}|}{|t-t^*_{k,j}|}&\leq
{\sup_{|\bar t-t^*_{k,j}|<  c_0 |k|^{-\frac12}} |\dddot\vp_k(\bar t)|} \\
&\leq
{\sup_{|\bar t-t^*_{k,j}|<  c_0 |k|^{-\frac12}} |kf'(x_0(\bar t))\dot x_0(\bar t)|}\\
&\leq C\|f'\|_{C^0(K_1)}|k|
\end{align*}
for $|t-t^*_{k,j}|<c_0|k|^{-1/2}$.
Likewise,
\begin{equation}\label{boundsh}
|\dot h(t)|+|\ddot h(t)|+|\dddot h(t)|\leq C|k|
\end{equation}
for all~$|t-t^*_{k,j}|< c_0|k|^{-\frac12}$, with $C$ a constant that
depends on $\|f\|_{C^3(K_1)}$. In particular, if the constant~$c_0$ is small enough (but independent of~$k$
and~$\ep$), one can define a real function~$\tau=\tau(t)$ as
\[
\tau:=|h(t)|^{\frac12}(t-t^*_{k,j})\,,
\]
with
\begin{equation}\label{newcI}
|t-t^*_{k,j}|<c_0|k|^{-1/2}\,.
\end{equation}
In terms of this variable, one can write
\[
\vp_k= \al_{k,j}+\si_{k,j} \tau^2\,,
\]
where the multiplicative factor $\si_{k,j}:= kf_{k,j}/|kf_{k,j}|$ is plus or
minus one. Moreover, the interval~\eqref{newcI}
can be described in terms
of the new variable as $\tau\in\cI$, where the new interval satisfies
\[
\cI\subset\{|\tau|< c_1\}
\]
with $c_1$ a constant independent of~$\ep$ and~$k$.

It is easy to see that both $\dot\tau$ and
\[
V:=\frac1{|\dot\tau|}=
{\frac{2|h|^{3/2}}{2|h|^2+\dot h h (t-t^*_{k,j})}}
\]
do not vanish in the interval $|t-t^*_{k,j}|<c_0|k|^{-\frac12}$
if~$c_0$ is small enough (independently of~$k$ and~$\ep$), and that
\begin{equation}\label{Vfkj}
V|_{t=t^*_{k,j}}=2^{1/2}|kf_{k,j}|^{-1/2}\,.
\end{equation}
In fact,
using that
\[
\frac{dV}{d\tau} =\frac{\dot V}{\dot\tau}
\]
and an analogous formula for the second derivative, it is easy to see
that in the region $|t-t^*_{k,j}|<c_0|k|^{-1/2}$ the function~$V$ is
bounded as
\begin{equation}\label{boundsV}
|V|<C|k|^{-\frac12}\,,\qquad \bigg|\frac {dV}{d\tau}\bigg|<C|k|^{-1}\,,\qquad \bigg|\frac {d^2V}{d\tau^2}\bigg|<C|k|^{-\frac32}\,.
\end{equation}
Again, $C$ does not depend on~$k$ or~$\ep$.

Still denoting by $\dot x_0$, $\chi_{k,j}^1$ and $q_k({\dot x_0})$ the expression of these
functions in terms of the new variable~$\tau$, with some abuse of
notation, notice that
\begin{equation}\label{boundtau}
\bigg|\frac{ d\chi_{k,j}^1}{d\tau}\bigg| + \frac{|k|^{5/2}}M \bigg|\frac{
  dq_k({\dot x_0})}{d\tau}\bigg|+
{|k|^{1/2}}
\bigg|\frac{ d\dot x_0}{d\tau}\bigg|+ \bigg|\frac{ d^2\chi_{k,j}^1}{d\tau^2}\bigg| + \frac{k^2}M\bigg|\frac{
  d^2q_k({\dot x_0})}{d\tau^2}\bigg|+
{|k|}
  \bigg|\frac{ d^2\dot x_0}{d\tau^2}\bigg|<C
\end{equation}
with a constant independent of~$\ep$ or~$k$. This follows from the
bounds~\eqref{decayqk}--\eqref{boundsh} after writing the derivatives with respect
to~$\tau$ in terms of derivatives with respect to~$t$. The integral for~$m_{k,j}^1$ can
therefore be written as
\begin{align*}
m_{k,j}^1&= e^{i\al_{k,j}/\ep}\int_{-c_1}^{c_1} V\dot x_0\chi_{k,j}^1
           q_k({\dot x_0})\, e^{i\si_{k,j}\tau^2/\ep}\, d\tau\,.
\end{align*}
Notice that~$V$ appears here as the Jacobian of the change of
variables.

To evaluate this integral, let us decompose it as
\[
m_{k,j}^1=e^{i\al_{k,j}/\ep} (m_{k,j}^3+ m_{k,j}^4)\,,
\]
where
\begin{align*}
m_{k,j}^3&= [V\dot x_0\chi_{k,j}^1
           q_k({\dot x_0})]|_{\tau=0} \int_{-c_1}^{c_1}
           e^{i\si_{k,j}\tau^2/\ep}\, d\tau= \frac{2^{1/2}q_k({\dot x_0}(t^*_{k,j}))}{k|k|^{1/2} |f_{k,j}|^{1/2}}\int_{-c_1}^{c_1}
           e^{i\si_{k,j}\tau^2/\ep}\, d\tau\,,\\
m_{k,j}^4&= \int_{-c_1}^{c_1} \big[V\dot x_0\chi_{k,j}^1
           q_k({\dot x_0})-[V\dot x_0\chi_{k,j}^1
           q_k({\dot x_0})]|_{\tau=0}\big]\, e^{i\si_{k,j}\tau^2/\ep}\, d\tau\,.
\end{align*}
It is clear that the first integral can be evaluated explicitly modulo
higher order corrections using Fresnel integrals:
\begin{align}\notag
\int_{-c_1}^{c_1} e^{i\si_{k,j}\tau^2/\ep}\, d\tau&=
\ep^{\frac12}\int_{-\infty}^\infty
e^{i\si_{k,j}\tau^2}d\tau-\ep^{\frac12}\int_{|\tau|>c_1/\sqrt\ep}e^{i\si_{k,j}\tau^2}d\tau\\
&=
({\pi\ep})^{\frac12}e^{i\si_{k,j}\pi/4}+ \cR_2\,,\label{inte}
\end{align}
with $|\cR_2|<C\ep$. To estimate~$m^4_{k,j}$, we use Taylor's
formula to write
\[
V\dot x_0\chi_{k,j}^1
           q_k(\dot x_0)-[V\dot x_0\chi_{k,j}^1
           q_k(\dot x_0)]|_{\tau=0} =: \tau \, F(\tau)\,,
\]
where the bounds~\eqref{boundsV}-\eqref{boundtau} ensure that
\[
|k|^3|F(\tau)|+ |k|^{5/2}|F'(\tau)|<CM
\]
in the domain of integration. It then follows that
\begin{align*}
|m_{k,j}^4|& = \bigg|\int_{-c_1}^{c_1} \tau F(\tau)
             e^{i\si_{k,j}\tau^2/\ep}\, d\tau\bigg|\\
&=\frac\ep2\bigg|\int_{-c_1}^{c_1} F(\tau)
             \frac d{d\tau} e^{i\si_{k,j}\tau^2/\ep}\, d\tau\bigg|\\
&=\frac\ep2\bigg|\int_{-c_1}^{c_1} F'(\tau)
             e^{i\si_{k,j}\tau^2/\ep}\, d\tau\bigg|\\
&\leq \frac\ep2 \int_{-c_1}^{c_1} |F'(\tau)|\,d\tau\\
&\leq CM\ep|k|^{-5/2} \,.
\end{align*}
The expression for $m_{k,j}$ in the lemma then follows  upon {recalling that $\dot x_0(t_{k,j}^*)=1/k$}.

\subsubsection*{Second case: $|t^*_{k,j}|> T_0$.} In this case we take
a small constant~$c_0<\eta$ independent of~$\ep$ and~$k$ and define a
function~$\chi^1_{k,j}(t)$ that is equal to~1 if
$|t-t^*_{k,j}|<c_0/2$, vanishes identically for
$|t-t^*_{k,j}|>c_0$, and is bounded as
\begin{equation}\label{boundschi1bis}
\sup_{t\in\RR}\Big(|\chi_{k,j}^1(t)|+|\dot\chi^1_{k,j}(t)|
{+|\ddot\chi^1_{k,j}(t)|}
\Big)<C\,,
\end{equation}
where $C$ is a uniform constant. We set
\[
\chi^2_{k,j}(t):=\chi_{k,j}(t)-\chi_{k,j}^1(t)\,,
\]
which allows us to write
\[
m_{k,j}=m_{k,j}^1+m_{k,j}^2
\]
with
\begin{align*}
m_{k,j}^l:=\int_{-\infty}^\infty \dot x_0(t)\,
          q_k(\dot x_0(t))\,e^{i\vp_k(t)/\ep}\, \chi_{k,j}^l(t)\,
           dt
\end{align*}
and $l=1,2$.

We argue as in the first case to write
\begin{align*}
m_{k,j}^2&=-i\ep \int_{-\infty}^\infty \frac{\dot
  x_0\chi_{k,j}^2q_k(\dot x_0(t) )}{\dot\vp_k}\frac d{dt} e^{i\vp_k/\ep}\,
           dt\\
&=i\ep \int_{-\infty}^\infty \frac d{dt} \bigg( \frac{\dot
  x_0\chi_{k,j}^2q_k(\dot x_0(t))}{\dot\vp_k}\bigg)\, e^{i\vp_k/\ep}\,
           dt\,.
\end{align*}
The bound~\eqref{x0t} together with the estimates for~$\vp_k(t)$
presented in
item~(ii)
 of Lemma~\ref{L.cCk} then permit to estimate
$m^2_{k,j}$ using the same argument as in the first case:
\begin{align*}
|m_{k,j}^2|&\leq \frac{C\ep}{|k|}
             \bigg[\|q_k\|_{C^1(K_2)}\int_{\frac{c_0}2<|s|<{2}\eta}\frac{ds}s +
             \|q_k\|_{C^0(K_2)}\int_{\frac{c_0}2<|s|<{2}\eta}\frac{ds}{s^2}\bigg]\\
&\leq \frac{C\ep}{|k|}\|q_k\|_{C^1(K_2)}\,.
\end{align*}

Let us now compute the remaining integral, $m_{k,j}^1$. Again using Taylor's formula
at~$t^*_{k,j}$ and the fact that the absolute value of
\[
f_{k,j}:=f(x_0(t^*_{k,j}))=\ddot x_0(t^*_{k,j})
\]
is bigger than $c/|k|$ by Lemma~\ref{L.cCk}
{(item~(ii))},
one can write
\[
\vp_k(t)=\al_{k,j}+ h(t)\, (t-t^*_{k,j})^2\,,
\]
where $\al_{k,j}:= \vp_k(t^*_{k,j})$,
{$h(t^*_{k,j})=\frac12 kf_{k,j}$}
and
\begin{equation*}%\label{boundsh}
{\sup_{|t-t^*_{k,j}|<c_0}\big[|\dot h (t)|+|\ddot h(t)|+|\dddot h(t)|\big]\leq C\|f\|_{C^3(K_1)}}
\end{equation*}
just as in Equation~\eqref{x0t}. Notice that, in contrast to the first case, no powers of~$k$ appear in the
bounds. Hence, if the uniform constant~$c_0$ is small enough, one can define a real function~$\tau=\tau(t)$ as
\[
\tau:=|h(t)|^{\frac12}(t-t^*_{k,j})\,,
\]
with $|t-t^*_{k,j}|<c_0$, and one can then write
\[
\vp_k= \al_{k,j}+\si_{k,j} \tau^2\,,
\]
where the multiplicative factor $\si_{k,j}:= kf_{k,j}/|kf_{k,j}|$ is plus or
minus one. Moreover, the interval $|t-t^*_{k,j}|<c_0$ can be described in terms
of the new variable as $\tau\in\cI$, where the new interval satisfies
\[
\cI\subset\{|\tau|< c_1\}
\]
with $c_1$ a constant independent of~$\ep$ and~$k$. Analogously to the
first case, one can show that the Jacobian of this change of variables
is
\[
V:=
{\frac{2|h|^{3/2}}{2h^2+\dot h h (t-t^*_{k,j})}}\,,
\]
which does not vanish in the interval $|t-t^*_{k,j}|<c_0$
for a small enough uniform constant~$c_0$. Notice that~$V$ is well
defined because the denominator does not vanish in the above interval. Its value at $t^*_{k,j}$ is
\begin{equation*}%\label{Vfkj}
V|_{t=t^*_{k,j}}=2^{1/2}|kf_{k,j}|^{-1/2}\,.
\end{equation*}
Notice that the fact that $f_{k,j}=\ddot x_0(t^*_{k,j})$ and the
asymptotics~\eqref{x0t} show that, contrary to what happened in the
first case, here $V$~is uniformly bounded in the interval under
consideration in the sense that
\[
\frac1C<|V|<C\,,\qquad \bigg|\frac{dV}{d\tau}\bigg| + \bigg|\frac{d^2V}{d\tau^2}\bigg|<C\,.
\]
{The estimates in \eqref{boundtau} are then replaced by
\[
\bigg|\frac{ d\chi_{k,j}^1}{d\tau}\bigg| + \frac{|k|^{3}}M \bigg|\frac{
  dq_k({\dot x_0})}{d\tau}\bigg|+
{|k|}
\bigg|\frac{ d\dot x_0}{d\tau}\bigg|+ \bigg|\frac{ d^2\chi_{k,j}^1}{d\tau^2}\bigg| + \frac{|k|^3}M\bigg|\frac{
  d^2q_k({\dot x_0})}{d\tau^2}\bigg|+
{|k|}
  \bigg|\frac{ d^2\dot x_0}{d\tau^2}\bigg|<C
\]
with a constant independent of~$\ep$ or~$k$.}

To compute~$m_{k,j}^1$, we can use the same trick, obtaining
completely analogous formulas. Specifically, just as in the first case
one can decompose
\[
m_{k,j}^1=e^{i\al_{k,j}/\ep} (m_{k,j}^3+ m_{k,j}^4)\,,
\]
where
\begin{align*}
m_{k,j}^3&= [V\dot x_0\chi_{k,j}^1
           q_k(\dot x_0(t))]|_{\tau=0} \int_{-c_1}^{c_1}
           e^{i\si_{k,j}\tau^2/\ep}\, d\tau= \frac{2^{1/2}q_k(\dot x_0 (t^*_{k,j}))}{k|k|^{1/2} |f_{k,j}|^{1/2}}\int_{-c_1}^{c_1}
           e^{i\si_{k,j}\tau^2/\ep}\, d\tau\,,\\
m_{k,j}^4&= \int_{-c_1}^{c_1} \big[V\dot x_0\chi_{k,j}^1
           q_k(\dot x_0)-[V\dot x_0\chi_{k,j}^1
           q_k(\dot x_0)]|_{\tau=0}\big]\, e^{i\si_{k,j}\tau^2/\ep}\, d\tau\,.
\end{align*}
The first integral can be evaluated explicitly modulo
higher order corrections using the formula~\eqref{inte}
{and recalling that $\dot x_0(t_{k,j}^*)=1/k$}, obtaining the
same formula as in the first case. The second integral can be shown to be small upon using Taylor's
formula to write
\[
[V\dot x_0\chi_{k,j}^1
           q_k(\dot x_0)-[V\dot x_0\chi_{k,j}^1
           q_k(\dot x_0)]|_{\tau=0} =: \tau \, F(\tau)
\]
with
\[
k^4\big(|F(\tau)|+ |F'(\tau)|\big)<CM
\]
in the domain of integration. Arguing as in the first case one then finds
\begin{align*}
|m_{k,j}^4| \leq CM\ep
{k^{-4}}
\,,
\end{align*}
and the lemma follows.
\end{proof}

\subsubsection*{Step 4: Formula for the displacement function}

To
complete the proof of the theorem, we just need to combine the
relation between the displacement and Melnikov functions, captured in
Proposition~\ref{P.Melnikov}, the formula for the Melnikov function in
terms of the numbers~$m_k$ and an error (Equations~\eqref{error1}--\eqref{mks}), and the expressions for~$m_{k,j}$
computed in Lemmas~\ref{L.mk0}-\ref{L.mkj}.
Using the notation
$$
\si^*:=kf(x_0(t^*))/|kf(x_0(t^*))|\,,
$$
this immediately yields
\begin{multline*}
\cD(t_0)=\ep^r\sum_{k\in\ZZ\backslash\{0\}
}\Real \bigg[ \sum_{t^* \in \cC_k} \bigg(\frac{2\pi \ep}{|f(x_0(t^*))|}\bigg)^{\frac12} \frac{q_k(k^{-1})}{k|k|^{1/2}}
\,
  e^{ikx_0(t^*)-it^*+i\si^*\frac\pi4 -i\frac{t_0}{\ep}}\bigg]\\
+O(\ep^{r+1}+\ep^{2r-2})\,,
\end{multline*}
which amounts to the formula provided in the statement of
Theorem~\ref{T.Melnikov}.

\section{An example: Dynamics in a rapidly oscillating electromagnetic
  field}
\label{S.ex}

A context where rapid spacetime oscillations appear naturally is in
dynamical systems subject to a perturbation controlled by a wave equation. Roughly
speaking, the reason is that the wave equation
\begin{equation}\label{wave}
\frac{\pd^2 f}{\pd t^2}= \frac{\pd^2 f}{\pd x^2}
\end{equation}
implies that the time frequencies and the space frequencies are equal,
which corresponds to the situation modeled by the
perturbation~\eqref{g} that we have considered in this paper. A prime
example of physical phenomena controlled by the wave equation are
electromagnetic fields, which are described by the Maxwell equations
in vacuum:
\begin{gather*}
\frac{\partial{E}}{\partial t}  =  \curl{B}\,,\qquad  \Div{E}=0\,,   \\
\frac{\partial{B}}{\partial t}   =  -\curl{E}\,,\qquad   \Div B=0\,.
\end{gather*}
The three-dimensional time-dependent vector fields $E(x,y,z,t)$
and $B(x,y,z,t)$ are the electric and magnetic field, respectively. For the benefit of the reader we will next
present a simple physical model where one can encounter this kind of
rapid spacetime oscillations.

Consider as the unperturbed system a charged particle moving in the
electric field generated by a harmonic time-independent potential, such as
\[
E_0:= -\nabla V\,,\qquad V(x,y,z):= \cos x \cosh z\,.
\]
With $X:=(x,y,z)$, the equations of motion are then $\ddot X=-\nabla
V$, that is,
\begin{align*}
\ddot x&= \sin x\cosh z\,,\qquad \ddot y=0\,,\qquad \ddot z=-\cos
         x\sinh z\,.
\end{align*}
The subset of the phase space
\begin{equation}\label{invset}
\{ (X,\dot X)\in\RR^6: y=\dot y=z=\dot z=0\}
\end{equation}
is obviously invariant, and the equation of motion there is just that
of a pendulum:
\[
\ddot x= \sin x\,.
\]

Let us now perturb this system by adding a small but rapidly oscillating
electromagnetic field of the form
\[
B_1:= \frac {\pd f(x,t)}{\pd t}\, (0,-z,y)\,, \qquad E_1:= \bigg(2f(x,t),-\frac {\pd
  f(x,t)}{\pd x}\,y, -\frac {\pd
  f(x,t)}{\pd x}\,z\bigg)\,.
\]
A straightforward computation shows that the electromagnetic fields
that we have considered satisfy the Maxwell equations if and only
if~$f(x,t)$ satisfies the wave equation~\eqref{wave}. A simple choice
for~$f$ is to set
\begin{equation}\label{f}
f(x,t):=\ep^r\sin\frac x\ep\cos\frac t\ep
\end{equation}
with $r>\frac52$. The perturbed
equations of motion in the total electromagnetic field,
\[
\ddot X= E_0+E_1+ \dot X\times B_1\,,
\]
can then be written as
\begin{align*}
\ddot x&= \sin x\cosh z- (y\dot y+z\dot z)\frac{\pd f(x,t)}{\pd t}+2f(x,t)\,,\\
 \ddot y&=-y\bigg(\frac{f(x,t)}{\pd x}-\dot x\frac{\pd f(x,t)}{\pd
          t}\bigg)\,,\\
 \ddot z&=-\cos x\sinh z-z\bigg(\dot x\frac{\pd f(x,t)}{\pd t}
  +\frac{\pd f(x,t)}{\pd x}\bigg)\,.
\end{align*}

Notice that the set~\eqref{invset} is invariant also for the perturbed
system for any choice of the function~$f(x,t)$. The equations of
motion on this subset reduce to the one and a half degrees of freedom
system
\[
\ddot x= \sin x+2f(x,t)\,,
\]
which with~$f(x,t)$ as in~\eqref{f} (or with many other rapidly oscillating
solutions of the wave equation~\eqref{wave}) is exactly of the form
studied in Theorem~\ref{T.Melnikov}. In this particular case, an easy
computation shows that the displacement function reads as
\[
\cD(t_0)=\ep^{r+\frac12}\left(\frac{16 \pi }{\sqrt{3}}\right)^{\frac 12}
\cos(\tfrac{13\pi}{12} +\log (2+\sqrt{3}))\sin\frac{t_0}\ep + O(\ep^{r+1}+\ep^{{2r-2}})\,.
\]
Now, by Corollary~\ref{C.splitting}, the perturbed
system with~$f$ given by~\eqref{f} is chaotic in the sense that it has positive topological
entropy for all small enough~$\ep>0$.

\section{Quasi-periodic perturbations}
\label{S.qp}

In this section we will study rapidly oscillating perturbations
that are not time-periodic, but quasiperiodic. While the periodic
case is a model for one and a half degrees of freedom systems,
the quasiperiodic case is a model for higher dimensional systems.
This model allows us to understand the
splitting of invariant manifolds of whiskered tori with~$n$
frequencies in a nearly integrable Hamiltonian system, and
hence it is relevant to the study of the mechanisms of diffusion.
Arguing as in Equation~\eqref{stab}, the corresponding a priori stable
problems can be understood in terms of perturbations of the form
\[
g(x,\dot x, t; \ep) := G(x,\ep \dot x) F \left(\frac{\omega t}{\ep} \right)\,,
\]
where $\omega \in \RR^n$ is a non-resonant frequency vector and
$F: \TT^n \rightarrow \RR$ is a periodic function of all its arguments.

As is well known, the
{quasiperiodic}
case is not well understood and
its analysis is usually rather involved. In fact,
the literature on this matter is rather scarce.
In general~\cite{DGJS97,Sau01,LMS03,Viejo9, DGG16}, one has only been able to prove the splitting of separatrices for very concrete systems (usually the quasiperiodically forced pendulum) and frequency vectors with very concrete arithmetic
properties (always in very low dimensions, with $n=2$ or~3, and
frequencies usually given by the golden ratio). Needless to say, all
the systems that have been considered are analytic.

% (beyond Diophantine
% conditions) of the frequencies
% are required to obtain (exponentially
% small) asymptotic estimates of the splitting of the manifolds.
% This was proved
% in~\cite{DGJS97} for the quasiperiodically forced pendulum,
% in the case
% $G(x):=\sin x$. Other interesting
% studies
% are \cite{Sau01,LMS03,Viejo9}, where exponentially
% small estimates for the transversality of the splitting were obtained, excluding some intervals
% of the perturbation parameter $\ep$. In all these works, the arguments
% work just for very low dimensions ($n=2$ or~3) and a few concrete values of the frequency vector (typically
% the golden ratio) as very specific properties of the continued
% fraction expansion of the frequency are crucially
% employed (see~\cite{DGG16} and references therein). Furthermore, these results make strong use of the
% facts that $G(x)=\sin x$ and assume that  $F(\theta)$ is an analytic
% function whose Fourier coefficients are all nonzero. This allows to characterize
% the dominant harmonics of the Melnikov potential in terms of the arithmetic
% properties of the frequency vector.

What we will show in this section is that, when the
{quasiperiodic}
perturbation features
fast oscillations in the space variable, the study of the
Melnikov function can be addressed using
the tools developed in Section~\ref{S.proof}. This permits to prove the splitting of separatrices
in fairly general contexts (provided, of course, that these
fast oscillations appear).
For concreteness we
will consider perturbations of the form
\begin{equation}\label{g-qp}
g (x,\dot x,t; \ep) := q\bigg(\frac x\ep,\dot x;\ep\bigg)
F\Big(\frac{\omega t}\ep\Big)\,,
\end{equation}
where $\omega \in \RR^n$ is a non-resonant frequency vector
and $F : \TT^n \rightarrow \RR$ is a trigonometric polynomial of zero mean,
which we write as
\[
F(\te)=:\sum_{m\in \cZ} F_m\, e^{im\cdot\te}
\]
with harmonics in a finite subset $\cZ \subset \ZZ^n \backslash\{0\}$.

\begin{theorem}\label{T.qp}
Let us consider the system~\eqref{eqx} with $r>\frac52$ and a
perturbation of the form~\eqref{g-qp}, where the functions~$f$ and~$g$
are of class~$C^3$. Suppose that the unperturbed system ($\ep=0$) has a hyperbolic equilibrium
with a homoclinic connection corresponding to a stable manifold and
an unstable manifold, which we describe through
an integral curve $x_0(t)$. Setting for each nonzero~$k\in \ZZ$ and each~$m\in\cZ$
\[
\cC_{k,m}:=\left\{ t^*\in\RR: \dot x_0(t^*) = - \frac{m \cdot \omega}{k} \right\}\,,
\]
assume moreover that $|f(x_0(t^*))|\neq0$ for all $t^*\in \cC_{k,m}$, $k\in\ZZ\backslash\{0\}$ and $m \in \cZ$.
Then the displacement function is
\[
\cD(t_0)=\ep^{r+\frac12}\sum_{m\in\cZ}\Big(
\cA_m \cos \frac{m \cdot \omega  t_0}{\ep}
+
\cB_m \sin \frac{m \cdot \omega  t_0}{\ep}
\Big) + O(\ep^{r+1}+\ep^{{2r-2}})\,,
\]
where $\cA_m$ and $\cB_m$ are the real constants
\begin{align*}
\cA_m&:= \sqrt{2\pi} \Real\left[F_m \sum_{
k\in\ZZ\backslash\{0\}
} \sum_{t^*\in\cC_k} \frac{ q_k(k^{-1})\, e^{i(k x_0(t^*)-t^*+\sigma^* \pi/4)}}{k|k|^{1/2}|f(x_0(t^*))|^{1/2}}\right] \,,\\
\cB_m&:= \sqrt{2\pi} \Imag\left[F_m \sum_{
k\in\ZZ\backslash\{0\}
} \sum_{t^*\in\cC_k} \frac{ q_k(k^{-1})\, e^{i(k x_0(t^*)-t^*+\sigma^* \pi/4)}}{k|k|^{1/2}|f(x_0(t^*))|^{1/2}}\right]
\end{align*}
and the cardinality of~$\cC_{k,m}$ is uniformly
bounded in~$k$.
\end{theorem}

% \begin{remark}
% It is easy to see that it suffices to consider the sum over~$k$ with
% $|k|\geq {|m\cdot \omega|}/{\max |\dot x_0|}$.
% \end{remark}

%\begin{remark}
%\textcolor{blue}{
%The tools presented in this paper can be readily extended to consider
%the so-called quasiperiodic case. Think for example in the perturbation
%\begin{equation}\label{g-qp2}
%g (x,\dot x,t; \ep) = q\bigg(\frac x\ep,\dot x;\ep\bigg)
%F(\omega_1 t, \omega_2 t, \ldots \omega_n t)\,,
%\end{equation}
%where $\omega \in \RR^n$ is a non-resonant frequency vector
%and $F : \TT^d \rightarrow \RR$. In the classical setting (where the
%variable $x$ does not oscillate fast), this problem is much more
%difficult to study than the periodic case (referencias), due to ...
%However, in the setting discussed in this paper, the quasiperiodic
%case can be considered with minor modifications. For example,
%if we consider
%\[
%\omega=(1/\ep,\gamma/\ep)\,, \qquad F(\theta_1,\theta_2)=\sin \theta_1 \sin \theta_2
%\]
%then the problem reduces to study the phase function
%\[
%\vp_k(t)=x_0(t)-\frac{1+\gamma}{k}
%\]
%bla, bla, bla.}
%\end{remark}
%

The proof of Theorem~\ref{T.qp} goes just as that of
Theorem~\ref{T.Melnikov}, mutatis mutandis.
One starts by introducing a new phase function
\[
\vp_{k,m}(t) := k  x_0(t)+ m\cdot \omega t
\]
that takes into account the
{quasiperiodic}
dependence (simply because
it depends on~$m\in\cZ$).
We then observe that Proposition~\ref{P.Melnikov} holds true when the perturbation is quasiperiodic.
The Melnikov function now
reads as
\[
\cM(t_0) = \ep^r \int_{-\infty}^\infty
\dot x_0(t) q\left( \frac{x_0(t)}{\ep},\dot x_0(t);\ep \right)
F\left(\frac{\omega (t+t_0)}{\ep}\right) dt\,,
\]
which one can approximate as $\cM(t_0)=\widetilde
\cM(t_0)+\cR_1(t_0)$. Here the leading term is
\[
\widetilde \cM(t_0) = \ep^r \int_{-\infty}^\infty
\dot x_0(t) \, q\left( \frac{x_0(t)}{\ep},\dot x_0(t);0 \right)
F\left(\frac{\omega (t+t_0)}{\ep}\right) dt\,,
\]
and the error is bounded as
\[
|\cR_1(t_0)| \leq \ep^r \|q(\cdot,\cdot,\ep)-q(\cdot,\cdot,0)\|_{C^0(K)}
\|F\|_{C^0(\TT^d)} \int_{-\infty}^\infty |\dot x_0(t)|dt \leq C \ep^{r+1}\,.
\]
The leading part of the Melnikov function is then written as
\[
\widetilde \cM(t_0)=\ep^r \sum_{m\in \cZ}\sum_{k=-\infty}^\infty F_m e^{i\frac{m\cdot \omega t_0}{\ep}}
M_{k,m}\,,
\]
where
\[
M_{k,m} :=
\int_{-\infty}^\infty \dot x_0(t) q_k(\dot x_0(t)) e^{i \frac{\vp_{k,m}(t)}{\ep}} dt\,.
\]
As~$m$ ranges over the finite set~$\cZ$ and the frequency vector~$\om$ is non-resonant, Lemma~\ref{L.cCk} can be readily extended to analyze the critical
points of the new phase function~$\vp_{k,m}$. An immediate consequence
of the proof, in fact, is that the sum over~$k$ can be restricted to
the set $|k|\geq |m \cdot \omega|/\max |\dot x_0|$. The details of the
argument are just as in the proof of Theorem~\ref{T.Melnikov}.

\section*{Acknowledgments}

A.E.\ and D.P.-S.\ are respectively supported by the ERC Starting Grants~633152 and~335079. A.L.\ is supported by the Knut och Alice Wallenbergs stiftelse KAW 2015.0365.
This work is supported in part by the ICMAT--Severo Ochoa grant
SEV-2015-0554.

% \appendix
% \section{Some stationary phase lemmas}\label{Appendix}

\bibliographystyle{amsplain}

\end{document}